\documentclass{amsart}
\usepackage{amsfonts}
\usepackage{amssymb,amsmath,amsthm,amscd,latexsym}
\usepackage{url,hyperref}
\usepackage{graphicx, color}
\usepackage{blkarray}
\usepackage{float}

\newtheorem{theorem}{Theorem}[section]
\newtheorem{lemma}[theorem]{Lemma}
\newtheorem{proposition}[theorem]{Proposition}
\newtheorem{corollary}[theorem]{Corollary}
\theoremstyle{definition}

\newtheorem{example}[theorem]{Example}

\theoremstyle{remark}
\newtheorem{remark}[theorem]{Remark}
\numberwithin{equation}{section}
\newcommand{\F}{\mathbb{F}}

\begin{document}
\title[A generalization of the four circulant construction]{New binary self-dual codes via a generalization of the four circulant
construction}
\author{Joe Gildea}
\address{Department of Mathematics, University of Chester,UK }
\email{j.gildea@chester.ac.uk}
\author{Ab\.id\.in Kaya}
\address{Department of Mathematics Education, Sampoerna University, 12780,
Jakarta, Indonesia} \email{abidin.kaya@sampoernauniversity.ac.id}
\author{Bahattin Yildiz\textsuperscript{*}}\thanks{*The Corresponding Author}
\address{Department of Mathematics \& Statistics, Northern Arizona
University, Flagstaff 86011, AZ} \email{bahattin.yildiz@nau.edu}

\begin{abstract}
In this work, we generalize the four circulant construction for
self-dual
codes. By applying the constructions over the alphabets $\mathbb{F}_2$, $%
\mathbb{F}_2+u\mathbb{F}_2$, $\mathbb{F}_4+u\mathbb{F}_4$, we were
able to obtain extremal binary self-dual codes of lengths 40, 64
including new extremal binary self-dual codes of length 68. More
precisely, 43 new extremal binary self-dual codes of length 68, with
rare new parameters have been constructed.
\end{abstract}

\maketitle

\section{Introduction}

Binary self-dual codes have generated a considerable amount of
interest in the literature for decades for their connections to many
other mathematical structures and applications. They have an upper
bound on their minimum distance, which is given by Conway and Sloane
in \cite{conway}, and is finalized by
Rains in \cite{Rains} as $d\leq 4\lfloor \frac{n}{24}\rfloor +6$ when $%
n\equiv 22\pmod{24}$ and $d\leq 4\lfloor \frac{n}{24}\rfloor +4$,
otherwise, where $n$ is the length of the self-dual code. Self-dual
codes meeting these bounds are called \textit{extremal}.

There is an extensive literature on constructions for extremal
binary self-dual codes. One of the main directions of research in
the literature has been to construct extremal binary self-dual codes
whose weight enumerators have new parameters, that were not known to
exist before. This comes from the works by Conway and Sloane in
\cite{conway} and Dougherty et al. in \cite{dougherty1} in which the
possible weight enumerators of all extremal self-dual codes of
lengths up to 100 were classified.

While the tools in constructing extremal binary self-dual codes may
differ from taking a special matrix construction, considering a
certain automorphism or the neighboring construction, in all of
these cases the final step is to do a computer search over a reduced
set of possible inputs. Using the afore-mentioned tools reduce the
search field considerably so that the search is now feasible to do
in a reasonable time.

For most known constructions for self-dual codes, one of the key
concepts is ``criculant" matrices. It is well-known that circulant
matrices are determined uniquely by their first rows and that they
commute in matrix multiplication. The double-circulant, bordered
double circulant and four-circulant constructions are some of the
well-known construction methods in the literature that make use of
circulant martrices. Through these constructions the search field
for a self-dual code of length $2n$ usually reduces to a constant
multiple of $2^n$, which makes it feasible to search for self-dual
codes of lengths up to 88 for example.

In this work, we will be considering a generalized version of the
four-circulant construction over the alphabets $\F_2$, $\F_2+u\F_2$
and $\F_4+u\F_4$ to construct extremal binary self-dual codes. Our
construction, in general, is different than the four-circulant
construction and we will be giving the comparative results. Using
this construction, we are able to construct many extremal binary
self-dual codes of lengths 40 and 64, and in particular we are able
to construct 43 new extremal binary self-dual codes of length $68$
with new weight enumerators in $W_{68,2}$. The exact parameters in
the weight enumerators are given in section 5.

The rest of the paper is organized as follows. In section 2, we give
the preliminaries on the alphabets to be used, special types of
matrices that we use in our constructions and the well known four
circulant construction. In section 3, we introduce our
generalization of the four circulant construction and we give
theoretical results about when they lead to self-dual codes as well
as their connection to the ordinary four-circulant construction. In
section 4 we give the numerical results about extremal binary
self-dual codes of lengths 40 and 64 that we obtain by a direct
application of our constructions over different alphabets together
with a comparison with the usual four-circulant construction. In
section 5, we apply the neighboring construction as well as
extensions to the codes obtained in section 4 to find new extremal
binary self-dual codes of length 68. We finish with concluding
remarks and directions for possible future research.

\section{Preliminaries}

Let $\mathcal{R}$ be a commutative Frobenius ring of characteristic
2. A
code $\mathcal{C}$ of length $n$ over $\mathcal{R}$ is an $\mathcal{R}$-submodule of $%
\mathcal{R}^{n}$. Elements of the code $\mathcal{C}$ are called codewords of $\mathcal{C}$. Let $%
x=\left( x_{1},x_{2},\ldots ,x_{n}\right) $ and $y=\left(
y_{1},y_{2},\ldots ,y_{n}\right) $ be two elements of
$\mathcal{R}^{n}$. The duality is understood in terms of the
Euclidean inner product; $\left\langle x,y\right\rangle _{E}=\sum
x_{i}y_{i}$. The dual $\mathcal{C}^{\bot }$ of the code
$\mathcal{C}$ is defined as
\begin{equation*}
\mathcal{C}^{\bot }=\left\{ x\in \mathcal{R}^{n} \mid \left\langle
x,y\right\rangle _{E}=0\text{ for all }y\in \mathcal{C}\right\} .
\end{equation*}%
We say that $\mathcal{C}$ is self-dual if
$\mathcal{C}=\mathcal{C}^{\bot }$.

Two self-dual binary codes of
dimension $k$ are said to be neighbors if their intersection has dimension $%
k-1$.

Let $\mathbb{F}_{4}=\mathbb{F}_{2}\left( \omega \right) $ be the
quadratic field extension of the binary field $\mathbb{F}_{2}
=\{0,1\}$, where $\omega ^{2}+\omega +1=0$. The ring
$\mathbb{F}_{4}+u\mathbb{F}_{4}$ defined via $u^{2}=0$ is a
commutative
binary ring of size $16$. We may easily observe that it is isomorphic to $%
\mathbb{F}_{2}\left[ \omega ,u\right] /\left\langle u^{2},\omega
^{2}+\omega +1\right\rangle $. The ring has a unique non-trivial
ideal $\left\langle
u\right\rangle =\left\{ 0,u,u\omega ,u+u\omega \right\} $. Note that $%
\mathbb{F}_{4}+u\mathbb{F}_{4}$ can be viewed as an extension of $\mathbb{F}%
_{2}+u\mathbb{F}_{2}$ and so we can describe any element of $\mathbb{F}_{4}+u%
\mathbb{F}_{4}$ in the form $\omega a+\bar{\omega}b$ uniquely, where
$a,b\in \mathbb{F}_{2}+u\mathbb{F}_{2}$.

\begin{equation*}
\begin{tabular}{lll}
$\left( \mathbb{F}_{4}+u\mathbb{F}_{4}\right) ^{n}$ &
$\underrightarrow{\ \
\ \psi _{\mathbb{F}_{4}+u\mathbb{F}_{4}\ \ \ }}$ & $\left( \mathbb{F}_{2}+u%
\mathbb{F}_{2}\right) ^{2n}$ \\
$\downarrow \varphi _{\mathbb{F}_{4}+u\mathbb{F}_{4}}$ &  & $\downarrow $ $%
\varphi _{\mathbb{F}_{2}+u\mathbb{F}_{2}}$ \\
$\mathbb{F}_{4}^{2n}$ & $\underrightarrow{\ \ \ \ \ \ \ \psi _{\mathbb{F}%
_{4}\ \ \ \ \ }}$ & $\mathbb{F}_{2}^{4n}$%
\end{tabular}%
\end{equation*}

Let us recall the following Gray Maps from \cite{gaborit,ling} and \cite%
{dougherty2};%
\begin{eqnarray*}
\psi _{\mathbb{F}_{4}} &:&a\omega +b\overline{\omega }\mapsto \left(
a,b\right) \text{, \ }a,b\in \mathbb{F}_{2}^{n} \\
\varphi _{\mathbb{F}_{2}+u\mathbb{F}_{2}} &:&a+bu\mapsto \left(
b,a+b\right)
\text{, \ }a,b\in \mathbb{F}_{2}^{n} \\
\psi _{\mathbb{F}_{4}+u\mathbb{F}_{4}} &:&a\omega +b\overline{\omega }%
\mapsto \left( a,b\right) \text{, \ }a,b\in \left( \mathbb{F}_{2}+u\mathbb{F}%
_{2}\right) ^{n} \\
\varphi _{\mathbb{F}_{4}+u\mathbb{F}_{4}} &:&a+bu\mapsto \left(
b,a+b\right) \text{, \ }a,b\in \mathbb{F}_{4}^{n}
\end{eqnarray*}%
Note that these Gray maps preserve orthogonality in the respective
alphabets, for the
details we refer to \cite{ling}. The binary codes $\varphi _{\mathbb{F}_{2}+u%
\mathbb{F}_{2}}\circ \psi _{\mathbb{F}_{4}+u\mathbb{F}_{4}}\left(
\mathcal{C} \right) $
and $\psi _{\mathbb{F}_{4}}\circ \varphi _{\mathbb{F}_{4}+u\mathbb{F}%
_{4}}\left( \mathcal{C}\right) $ are equivalent to each other. The
Lee weight of an element in $\mathbb{F}_{4}+u\mathbb{F}_{4}$ is
defined to be the Hamming weight of its binary image under any of
the previously mentioned compositions of the maps. A self-dual code
is said to be of Type II if the Lee weights of all codewords are
multiples of $4$, otherwise it is said to be of Type I.

\begin{proposition}
$($\cite{ling}$)$ Let $\mathcal{C}$ be a code over
$\mathbb{F}_{4}+u\mathbb{F}_{4}$.
If $\mathcal{C}$ is self-orthogonal, so are $\psi _{\mathbb{F}_{4}+u\mathbb{F}%
_{4}}\left( \mathcal{C}\right) $ and $\varphi
_{\mathbb{F}_{4}+u\mathbb{F}_{4}}\left(
\mathcal{C}\right) $. $\mathcal{C}$ is a Type I (resp. Type II) code over $\mathbb{F}_{4}+u%
\mathbb{F}_{4}$ if and only if $\varphi _{\mathbb{F}_{4}+u\mathbb{F}%
_{4}}\left( \mathcal{C}\right) $ is a Type I (resp. Type II)
$\mathbb{F}_{4}$-code, if and only if $\psi
_{\mathbb{F}_{4}+u\mathbb{F}_{4}}\left( \mathcal{C}\right) $ is a
Type I (resp. Type II) $\mathbb{F}_{2}+u\mathbb{F}_{2}$-code.
Furthermore, the minimum Lee weight of $\mathcal{C}$ is the same as
the minimum Lee weight of $\psi
_{\mathbb{F}_{4}+u\mathbb{F}_{4}}\left( \mathcal{C}\right) $ and $\varphi _{\mathbb{F}%
_{4}+u\mathbb{F}_{4}}\left( \mathcal{C}\right) $.
\end{proposition}

\begin{corollary}
Suppose that $\mathcal{C}$ is a self-dual code over
$\mathbb{F}_{4}+u\mathbb{F}_{4}$
of length $n$ and minimum Lee distance $d$. Then $\varphi _{\mathbb{F}_{2}+u%
\mathbb{F}_{2}}\circ \psi _{\mathbb{F}_{4}+u\mathbb{F}_{4}}\left(
\mathcal{C}\right) $
is a binary $\left[ 4n,2n,d\right] $ self-dual code. Moreover, $\mathcal{C}$ and $%
\varphi _{\mathbb{F}_{2}+u\mathbb{F}_{2}}\circ \psi _{\mathbb{F}_{4}+u%
\mathbb{F}_{4}}\left( \mathcal{C}\right) $ have the same weight
enumerator. If $\mathcal{C}$ is Type I (Type II), then so is
$\varphi _{\mathbb{F}_{2}+u\mathbb{F}_{2}}\circ \psi
_{\mathbb{F}_{4}+u\mathbb{F}_{4}}\left( \mathcal{C}\right) $.
\end{corollary}

In subsequent sections we will be writing tables in which vectors
with elements from the rings $\F_2+u\F_2$ and
$\mathbb{F}_{4}+u\mathbb{F}_{4}$ will appear. In order to avoid
writing long vectors with elements that can be confused with other
elements, we will be describing the elements of this ring in a
shorthand way, which will make the tables more compact.

For the elements of $\F_2+u\F_2$ we will use $0 \rightarrow 0,
1\rightarrow  1, u\rightarrow u$ and $1+u\rightarrow 3$.

For the elements of $\mathbb{F}_{4}+u\mathbb{F}_{4}$, we use the
ordered basis $\{u \omega, \omega, u, 1 \}$ to express the elements
of $\mathbb{F}_{4}+u\mathbb{F}_{4}$ as binary strings of length 4.
Then we will use the hexadecimal number system to describe each
element:

$0 \leftrightarrow 0000$, $1 \leftrightarrow 0001$, $2
\leftrightarrow 0010$, $3 \leftrightarrow 0011$, $4 \leftrightarrow
0100$, $5 \leftrightarrow 0101$, $6 \leftrightarrow 0110$, $7
\leftrightarrow 0111$, $8 \leftrightarrow 1000$, $9 \leftrightarrow
1001$, $A \leftrightarrow 1010$, $B \leftrightarrow 1011$,
$C \leftrightarrow 1100$, $D \leftrightarrow 1101$, $E \leftrightarrow 1110$, $F \leftrightarrow 1111$.\\

For example $1 + u \omega$ corresponds to $1001$, which is
represented by the hexadecimal $9$, while $\omega+u\omega$
corresponds to $1100$, which is represented by $C$.

We are going to use the following extension method for computational
results in the upcoming sections.

\begin{theorem}
\label{extension}$($\cite{frobenius}$)$ Let $R$ be a commutative
ring of characteristic $2$ with identity. Let $C$ be a self-dual
code over $R$ of length $n$ and $G=(r_{i})$ be a $k\times n$
generator matrix for $C$, where $r_{i}$ is the $i$-th row of $G$,
$1\leq i\leq k$. Let $c$ be a unit in $R$ such that $c^{2}=1$ and
$X$ be a vector in $R^{n}$ with $\left\langle
X,X\right\rangle =1$. Let $y_{i}=\left\langle r_{i},X\right\rangle $ for $%
1\leq i\leq k$. Then the following matrix%
\begin{equation*}
\left[
\begin{array}{cc|c}
1 & 0 & X \\ \hline
y_{1} & cy_{1} & r_{1} \\
\vdots  & \vdots  & \vdots  \\
y_{k} & cy_{k} & r_{k}%
\end{array}%
\right] ,
\end{equation*}%
generates a self-dual code $D$ over $R$ of length $n+2$.
\end{theorem}

\subsection{Special Matrices}

Circulant matrices play an important role in many applications. In
this section we briefly recall circulant matrices and its variations
in the form of reverse-circulant and $\lambda$-circulant matrices.
For more detailed information on circulant matrices we refer the
reader to \cite{circulant1}, \cite{circulant2} and the references
therein.

With $R$ a commutative ring with identity, let $\sigma$ be the
permutation on $R^n$ that corresponds to the right shift, i.e.
\begin{equation}
\sigma(a_1,a_2, \dots, a_n) = (a_n,a_1, \dots, a_{n-1}).
\end{equation}
A circulant matrix is a square matrix where each row is a
right-circular shift of the previous row. In other words, if
$\overline{r}$ is the first row, a typical circulant matrix is of
the form
\begin{equation}
\left[
\begin{array}{c}
\overline{r} \\ \hline \sigma(\overline{r}) \\ \hline
\sigma^2(\overline{r}) \\ \hline \vdots \\ \hline
\sigma^{n-1}(\overline{r})%
\end{array}
\right].
\end{equation}
It is clear that, with $T$ denoting the permutation matrix
corresponding to the $n$-cycle $(123...n)$, a circulant matrix with
first row $(a_1,a_2, \dots, a_n)$ can be expressed as a polynomial
in $T$ as:
\begin{equation*}
a_1I_n+a_2T+a_3T^2+ \cdots + a_nT^{n-1},
\end{equation*}
with $T^n=I_n$. This shows that circulant matrices commute.

A reverse-circulant matrix is a square matrix where each row is a
left-circular shift of the previous row. It is clear to see that if $%
\overline{r}$ is the first row, a reverse-circulant matrix is of the
form
\begin{equation}
\left[
\begin{array}{c}
\overline{r} \\ \hline \sigma^{-1}(\overline{r}) \\ \hline
\sigma^{-2}(\overline{r}) \\ \hline \vdots \\ \hline
\sigma^{-(n-1)}(\overline{r})%
\end{array}
\right].
\end{equation}

An $n\times n$ square matrix $A$ is called $\lambda $-circulant if
every row is a $\lambda $-cyclic shift of the previous one, in other
words $A $ is in the following form;
\begin{equation*}
\left(
\begin{array}{ccccc}
a_{1} & a_{2} & a_{3} & \cdots & a_{n} \\
\lambda a_{n} & a_{1} & a_{2} & \cdots & a_{n-1} \\
\lambda a_{n-1} & \lambda a_{n} & a_{1} & \cdots & a_{n-2} \\
\vdots & \vdots & \vdots & \ddots & \vdots \\
\lambda a_{2} & \lambda a_{3} & \lambda a_{4} & \cdots & a_{1}%
\end{array}%
\right) .
\end{equation*}

$\lambda$-circulant matrices are an immediate generalization of
circulant matrices and like circulant matrices, two
$\lambda$-circulant matrices also commute.

$\lambda$-reverse-circulant matrices can also be defined in exactly
the same way as an extension of reverse circulant matrices.

The following lemma gives us an important result that we will be
using in the upcoming sections.

\begin{lemma}
\cite{kyp}\label{lemma} Let $A$ and $C$ be $\lambda $-circulant
matrices. Then $C^{\prime }=CR$ is a $\lambda $-reverse-circulant
matrix~and it is symmetric. Here $R$ is the back-diagonal matrix.
Moreover, $AC^{\prime }-C^{\prime }A^{T}=0$. Equivalently,
$ARC^{T}-CRA^{T}=0$.\qquad
\end{lemma}

A special case of Lemma \ref{lemma} is as follows;

\begin{lemma}
\label{commute} Symmetric circulant matrices commute with reverse
circulant matrices.
\end{lemma}

\subsection{On four circulant construction}
The four-circulant construction, which was inspired by orthogonal
designs, was introduced in \cite{betsumiya}:
\begin{theorem}
\cite{betsumiya} \label{four} Let $A$ and $B$ be $n\times n$
circulant matrices over $\mathbb{F}_{p}$ such that
$AA^{T}+BB^{T}=-I_{n}$ then the
matrix%
\begin{equation*}
G=\left( I_{2n\ }%
\begin{array}{|cc}
A & B \\
-B^{T} & A^{T}%
\end{array}%
\right)
\end{equation*}%
generates a self-dual code over $\mathbb{F}_{p}$.
\end{theorem}

Recently, the four circulant construction was applied on $\mathbb{F}_{2}+u%
\mathbb{F}_{2}$ in \cite{karadeniz4}, which resulted in a new binary
self-dual code of length 64.

The following is a variation of the four circulant construction,
which was used in \cite{kyp} to obtain new extremal binary self-dual
codes.

\begin{theorem}
\cite{kyp} \label{modified}Let $\lambda $ be a unit of the
commutative
Frobenius ring $\mathcal{R}$, $A$ be a $\lambda $-circulant matrix and $B$ be a $%
\lambda $-reverse-circulant matrix with $AA^{T}+BB^{T}=-I_{n}$ then
the matrix
\begin{equation*}
G=\left( I_{2n\ }%
\begin{array}{|cc}
A & B \\
-B & A%
\end{array}%
\right)
\end{equation*}%
generates a self-dual code $\mathcal{C}$ over $\mathcal{R}$.
\end{theorem}

\section{A generalization of the four circulant construction\label{theory}}

In this section, we give a generalization of the four circulant
construction. We also propose two specific variations of the
construction.

\begin{theorem}
\label{general}Let $\mathcal{R}$ be a commutative Frobenius ring of
characteristics $2$, $A$ and $B$ be circulant matrices and $C$ be a
reverse
circulant matrix. Then the code generated by%
\begin{equation*}
G:=\left( \enspace I_{2n}\enspace%
\begin{array}{|cc}
A & B+C \\
B^{T}+C & A^{T}%
\end{array}%
\right)
\end{equation*}%
is self-dual when $AA^{T}+BB^{T}+C^{2}=I_{n}$ and $AC=CA$.
\end{theorem}

\begin{proof}
Let $M:=\left(
\begin{array}{cc}
A & B+C \\
B^{T}+C & A^{T}%
\end{array}%
\right) $. We are to show that $MM^{T}=I_{2n}$ under the given
conditions. Indeed
\begin{eqnarray*}
MM^{T} &=&\left(
\begin{array}{cc}
AA^{T}+BB^{T}+BC+CB^{T}+C^{2} & AB+AC+BA+CA \\
B^{T}A^{T}+CA^{T}+A^{T}B^{T}+A^{T}C & B^{T}B+B^{T}C+CB+C^{2}+A^{T}A%
\end{array}%
\right) \\
&=&\left(
\begin{array}{cc}
AA^{T}+BB^{T}+C^{2} & AC+CA \\
CA^{T}+A^{T}C & B^{T}B+C^{2}+A^{T}A%
\end{array}%
\right) =\left(
\begin{array}{cc}
I_{n} & 0_{n} \\
0_{n} & I_{n}%
\end{array}%
\right) .
\end{eqnarray*}

The above equality holds because we have $AB=BA$ and
$A^{T}B^{T}=B^{T}A^{T}$ since circulant matrices commute and also
because by Lemma \ref{lemma} $BC+CB^{T}=0$ and $B^{T}C+CB=0$.
\end{proof}
We obtain the following corollary when $A$ is a symmetric circulant
matrix:
\begin{corollary}
\label{symmetric}Let $\mathcal{R}$ be a commutative Frobenius ring
of characteristic $2$, $A$ be a symmetric circulant matrix, $B$ be a
circulant matrix and $C$ be a reverse circulant matrix. Then the
code generated by
\begin{equation*}
G:=\left( \enspace I_{2n}\enspace%
\begin{array}{|cc}
A & B+C \\
B^{T}+C & A%
\end{array}%
\right)
\end{equation*}%
is a self-dual code over $\mathcal{R}$ whenever
$A^{2}+BB^{T}+C^{2}=I_{n}$.
\end{corollary}

\begin{proof}
It follows by Theorem \ref{general} and Lemma \ref{commute}.
\end{proof}

We may also propose another special case of Theorem \ref{general}.

\begin{corollary}
\label{czero}Let $\mathcal{C}$ be a self-dual four circulant code of length $4n$ over $%
\mathcal{R}$ (of characteristic 2) generated by
\begin{equation*}
G:=\left( \enspace I_{2n}\enspace%
\begin{array}{|cc}
A & B \\
B^{T} & A^{T}%
\end{array}%
\right) .
\end{equation*}%
Then for any reverse circulant matrix $C$, which commutes with $A$
and
satisfies $C^{2}=0$, the matrix%
\begin{equation*}
\left( \enspace I_{2n}\enspace%
\begin{array}{|cc}
A & B+C \\
B^{T}+C & A^{T}%
\end{array}%
\right)
\end{equation*}%
generates a self-dual code $\mathcal{C}^{\prime }$.
\end{corollary}

Corollary \ref{czero} allows us to reduce the size of the search
field for that specific variation. We may consider a four circulant
code and search for reverse circulant matrices $C$ under the
restrictions.

\begin{example}
Let $n=7$ and $\mathcal{C}$ be the four circulant code where $A=I_{7}$ and $%
B=0_{7}$, i.e. $r_{A}=\left( 1,0,0,0,0,0,0\right) $ and
$r_{B}=\left(
0,0,0,0,0,0,0\right) $. The code $\mathcal{C}$ is binary self-dual with parameters $%
\left[ 28,14,2\right]$. Let $C$ be the reverse circulant matrix with
first row $r_{C}=\left( 1110100\right) $ which satisfies
$C^{2}=0_{7}$ and obviously it commutes with $A$. Then the code
$\mathcal{C}^{\prime }$ obtained by Corollary \ref{czero} is an
extremal binary self-dual $\left[ 28,14,6\right]$ code with an
automorphism group of order $2^{6}\times 3\times 7$.
\end{example}

\section{Computational Results}

In this section, we provide examples to demonstrate the
effectiveness of the methods introduced in Section \ref{theory}. We
also compare the methods with the well known four circulant
construction for various lengths over the alphabets
$\mathbb{F}_{2}$, $\mathbb{F}_{2}+u\mathbb{F}_{2}$ and
$\mathbb{F}_{4}+u\mathbb{F}_{4}$ .
\subsection{Comparison of the methods over $\mathbb{F}_{2}$ for length 40}

We construct Type I self-dual codes of length 40 by the four
circulant construction and also by the methods given in Section
\ref{theory}. The weight enumerator of a singly even binary
self-dual code of parameters $[40,20,8]$ is in the following form:
\begin{eqnarray*}
W_{40} &=&1+\left( 125+16\beta \right) y^{8}+\left( 1664-64\beta
\right) y^{10}+\cdots ,0\leq \beta \leq 10
\end{eqnarray*}%
The existence of a code with $\beta=9$ is still an open problem.
There are codes for the other values. In Table \ref{tab:40-1}, we
list four circulant self-dual binary codes of length 40.
\begin{table}[H]
\caption{$[40,20,8]$ four circulant codes} \label{tab:40-1}
\begin{center}
\begin{tabular}{||c|c|c|c||c||}
\hline
$\mathcal{C}_{40,i}$ & $r_{A}$ & $r_{B}$ & $\left\vert Aut(\mathcal{C}%
_{40,i})\right\vert $ & $\beta $ in $W_{40}$ \\ \hline\hline
$\mathcal{C}_{40,1}$ & $\left( 0100001110\right) $ & $\left(
0100110011\right) $ & $2^{2}\times 5$ & $0$ \\ \hline
$\mathcal{C}_{40,2}$ & $\left( 0000110011\right) $ & $\left(
0010111001\right) $ & $2^{3}\times 5$ & $0$ \\ \hline
$\mathcal{C}_{40,3}$ & $\left( 0101100111\right) $ & $\left(
1111001011\right) $ & $2^{3}\times 3\times 5$ & $0$ \\ \hline
$\mathcal{C}_{40,4}$ & $\left( 1001000100\right) $ & $\left(
0101101101\right) $ & $2^{14}\times 3\times 5$ & $10$ \\ \hline
$\mathcal{C}_{40,5}$ & $\left( 1000000010\right) $ & $\left(
1101011101\right) $ & $2^{16}\times 3^{3}\times 5^{2}$ & $10$ \\
\hline
\end{tabular}%
\end{center}
\end{table}
The binary self-dual codes of length 40 constructed from the
construction given in Corollary \ref{symmetric} are given in Table
\ref{tab:40-2}. Since $A$ is symmetric circulant, we only list the
necessary entries of the first row $r_A$.
\begin{table}[H]
\caption{$[40,20,8]$ codes by Corollary \protect\ref{symmetric}}
\label{tab:40-2}
\begin{center}
\begin{tabular}{||c|c|c|c|c||c||}
\hline
$\mathcal{D}_{40,i}$ & $r_{A}$ & $r_{B}$ & $r_{C}$ & $\left\vert Aut(%
\mathcal{D}_{40,i})\right\vert $ & $\beta $ in $W_{40}$ \\
\hline\hline $\mathcal{D}_{40,1}$ & $\left( 110111\right) $ &
$\left( 0010001100\right) $ & $\left( 1110100101\right) $ & $2^{3}$
& $0$ \\ \hline $\mathcal{D}_{40,2}$ & $\left( 011010\right) $ &
$\left( 0010111000\right) $ & $\left( 1110010010\right) $ & $2^{4}$
& $0$ \\ \hline $\mathcal{D}_{40,3}$ & $\left( 011010\right) $ &
$\left( 0010011100\right) $ & $\left( 1010000100\right) $ &
$2^{3}\times 3$ & $0$ \\ \hline $\mathcal{D}_{40,4}$ & $\left(
100111\right) $ & $\left( 1010110011\right) $ & $\left(
0000100110\right) $ & $2^{6}$ & $0$ \\ \hline $\mathcal{D}_{40,5}$ &
$\left( 111001\right) $ & $\left( 1001011010\right) $ & $\left(
0010011001\right) $ & $2^{7}$ & $0$ \\ \hline $\mathcal{D}_{40,6}$ &
$\left( 010111\right) $ & $\left( 1001111001\right) $ & $\left(
0001011111\right) $ & $2^{2}$ & $1$ \\ \hline $\mathcal{D}_{40,7}$ &
$\left( 011001\right) $ & $\left( 0111000011\right) $ & $\left(
0111100001\right) $ & $2^{3}$ & $1$ \\ \hline $\mathcal{D}_{40,8}$ &
$\left( 011010\right) $ & $\left( 0010000110\right) $ & $\left(
1100000000\right) $ & $2^{4}$ & $1$ \\ \hline $\mathcal{D}_{40,9}$ &
$\left( 110001\right) $ & $\left( 1011000011\right) $ & $\left(
0100110100\right) $ & $2^{3}$ & $2$ \\ \hline $\mathcal{D}_{40,10}$
& $\left( 000110\right) $ & $\left( 0100010011\right) $ & $\left(
1100001000\right) $ & $2^{6}$ & $2$ \\ \hline $\mathcal{D}_{40,11}$
& $\left( 111001\right) $ & $\left( 1100101001\right) $ & $\left(
0100111110\right) $ & $2^{9}\times 3^{2}$ & $2$ \\ \hline
$\mathcal{D}_{40,12}$ & $\left( 110100\right) $ & $\left(
1101111011\right) $ & $\left( 0011111001\right) $ & $2^{13}$ & $2$
\\ \hline $\mathcal{D}_{40,13}$ & $\left( 001010\right) $ & $\left(
1010000111\right) $ & $\left( 1111110011\right) $ & $2^{5}$ & $4$
\\ \hline $\mathcal{D}_{40,14}$ & $\left( 011000\right) $ & $\left(
0101010110\right) $ & $\left( 0000001100\right) $ & $2^{6}$ & $4$
\\ \hline $\mathcal{D}_{40,15}$ & $\left( 100111\right) $ & $\left(
0000110011\right) $ & $\left( 1010011001\right) $ & $2^{7}$ & $4$
\\ \hline $\mathcal{D}_{40,16}$ & $\left( 000010\right) $ & $\left(
0011111000\right) $ & $\left( 0100101010\right) $ & $2^{8}$ & $6$
\\ \hline $\mathcal{D}_{40,17}$ & $\left( 000010\right) $ & $\left(
1011101000\right) $ & $\left( 0100101010\right) $ & $2^{8}\times 3$
& $6$ \\ \hline $\mathcal{D}_{40,18}$ & $\left( 010100\right) $ &
$\left( 1000111111\right) $ & $\left( 0100100101\right) $ & $2^{15}$
& $10$ \\ \hline $\mathcal{D}_{40,19}$ & $\left( 110100\right) $ &
$\left( 0101010101\right) $ & $\left( 1001110010\right) $ &
$2^{14}\times 3\times 5$ & $10$ \\ \hline $\mathcal{D}_{40,20}$ &
$\left( 101010\right) $ & $\left( 1000100010\right) $
& $\left( 0101111101\right) $ & $2^{16}\times 3^{3}\times 5^{2}$ & $10$ \\
\hline
\end{tabular}%
\end{center}
\end{table}
We apply Corollary \ref{czero} to $\mathcal{C}_{40,4}$ from Table
\ref{tab:40-1}. In other words, we fix the circulant matrices $A$
and $B$ and search for reverse circulant matrices which satisfy the
given conditions. The results are tabulated in Table \ref{tab:40-3}.
The results have shown the method to be quite effective.
\begin{table}[H]
\caption{$[40,20,8]$ codes by Corollary \protect\ref{czero} for $\mathcal{C}%
_{40,4}$ } \label{tab:40-3}
\begin{center}
\begin{tabular}{||c|c|c||c||}
\hline
$\mathcal{E}_{40,i}$ & $r_{C}$ & $\left\vert Aut(\mathcal{E}%
_{40,i})\right\vert $ & $\beta $ in $W_{40}$ \\ \hline\hline
$\mathcal{E}_{40,1}$ & $\left( 1101011010\right) $ & $2^{3}$ & $0$
\\ \hline $\mathcal{E}_{40,2}$ & $\left( 1100011000\right) $ &
$2^{2}$ & $2$ \\ \hline $\mathcal{E}_{40,3}$ & $\left(
0111001110\right) $ & $2^{3}$ & $2$ \\ \hline $\mathcal{E}_{40,3}$ &
$\left( 0100001000\right) $ & $2^{8}$ & $2$ \\ \hline
$\mathcal{E}_{40,4}$ & $\left( 0011100111\right) $ & $2^{13}$ & $2$
\\ \hline
$\mathcal{E}_{40,5}$ & $\left( 0111101111\right) $ & $2^{11}$ & $4$ \\
\hline $\mathcal{E}_{40,6}$ & $\left( 1010010100\right) $ & $2^{8}$
& $6$ \\ \hline $\mathcal{E}_{40,7}$ & $\left( 1101011010\right) $ &
$2^{8}\times 3$ & $6$
\\ \hline
$\mathcal{E}_{40,8}$ & $\left( 1000110001\right) $ & $2^{15}$ & $10$ \\
\hline
$\mathcal{E}_{40,9}$ & $\left( 1111011110\right) $ & $2^{16}$ & $10$ \\
\hline
\end{tabular}%
\end{center}
\end{table}

\subsection{Comparison of the methods over $\mathbb{F}_{2}$ for length 64}

There are two possibilities for the weight enumerators of extremal
Type I self-dual codes of length $64$ (hence of parameters $\left[
64,32,12\right]$) (\cite{conway}):
\begin{eqnarray*}
W_{64,1} &=&1+\left( 1312+16\beta \right) y^{12}+\left(
22016-64\beta
\right) y^{14}+\cdots , \\
W_{64,2} &=&1+\left( 1312+16\beta \right) y^{12}+\left(
23040-64\beta \right) y^{14}+\cdots
\end{eqnarray*}%
where $\beta $ and $\gamma $ are parameters.

Self-dual four circulant $\left[ 64,32,12\right] _{2}$ type I codes
exist for weight enumerators $\beta =0,8,16,24,32,40,48,56,64$ and $72$ in $%
W_{64,2}$. We provide codes that we obtained from Corollary
\protect\ref{symmetric} in Table \ref{tab:64-1}. The results show
that the limited version of the generalized four circulant
construction gives some codes which do not have four circulant
representation (The ones with $\beta = 4,10,12,13,17,18, 20,28,34$.)

\begin{table}[H]
\caption{Type I extremal self-dual codes of length 64 by Corollary
\protect\ref{symmetric}} \label{tab:64-1}
\begin{center}
\begin{tabular}{||c|c|c|c|c||c||}
\hline $\mathcal{C}_{i}$ & $r_{A}$ & $r_{B}$ & $r_{C}$ & $\left\vert
Aut(\mathcal{ C}_{i})\right\vert $ & $\beta $ in $W_{64,2}$ \\
\hline\hline $\mathcal{C}_{1}$ & $\left( 101001111\right) $ &
$\left( 0001111001100011\right) $ & $\left( 1101001110011001\right)
$ & $2^{3}$ & $4$
\\ \hline
$\mathcal{C}_{2}$ & $\left( 001101011\right) $ & $\left(
0101001000001001\right) $ & $\left( 0101011110111101\right) $ &
$2^{4}$ & $8$
\\ \hline
$\mathcal{C}_{3}$ & $\left( 111101110\right) $ & $\left(
0011110000111101\right) $ & $\left( 0101011110111101\right) $ &
$2^{5}$ & $8$
\\ \hline
$\mathcal{C}_{4}$ & $\left( 111101011\right) $ & $\left(
0000100001000001\right) $ & $\left( 0100000011001101\right) $ & $2^{3}$ & $%
10 $ \\ \hline $\mathcal{C}_{5}$ & $\left( 101101111\right) $ &
$\left(
1000101000101000\right) $ & $\left( 0010111110100010\right) $ & $2^{4}$ & $%
12 $ \\ \hline $\mathcal{C}_{6}$ & $\left( 100110100\right) $ &
$\left(
0001001001111111\right) $ & $\left( 0101011110111101\right) $ & $2^{3}$ & $%
13 $ \\ \hline $\mathcal{C}_{7}$ & $\left( 100011101\right) $ &
$\left(
1010110001101011\right) $ & $\left( 1100111001100100\right) $ & $2^{4}$ & $%
16 $ \\ \hline $\mathcal{C}_{8}$ & $\left( 110001001\right) $ &
$\left(
0010111110100010\right) $ & $\left( 1100011011111000\right) $ & $2^{3}$ & $%
17 $ \\ \hline $\mathcal{C}_{9}$ & $\left( 000011110\right) $ &
$\left(
0000011101110111\right) $ & $\left( 0100000011001101\right) $ & $2^{3}$ & $%
18 $ \\ \hline $\mathcal{C}_{10}$ & $\left( 101101100\right) $ &
$\left(
1111001000100111\right) $ & $\left( 1100101101001011\right) $ & $2^{4}$ & $%
20 $ \\ \hline $\mathcal{C}_{11}$ & $\left( 000011100\right) $ &
$\left(
0010001101001111\right) $ & $\left( 0101011110111101\right) $ & $2^{3}$ & $%
24 $ \\ \hline $\mathcal{C}_{12}$ & $\left( 010111001\right) $ &
$\left(
1010000110110000\right) $ & $\left( 0100010001011111\right) $ & $2^{4}$ & $%
24 $ \\ \hline $\mathcal{C}_{13}$ & $\left( 010010011\right) $ &
$\left(
0000000100000101\right) $ & $\left( 0101011110111101\right) $ & $2^{4}$ & $%
28 $ \\ \hline $\mathcal{C}_{14}$ & $\left( 111011011\right) $ &
$\left(
0100000101111111\right) $ & $\left( 0001010001000001\right) $ & $2^{4}$ & $%
32 $ \\ \hline $\mathcal{C}_{15}$ & $\left( 011111000\right) $ &
$\left(
1000101000100010\right) $ & $\left( 0001001101000011\right) $ & $2^{3}$ & $%
\mathbf{34}$ \\ \hline
\end{tabular}%
\end{center}
\end{table}

\begin{remark}
The first code with weight enumerator $\beta =34$ in $W_{64,2}$ has
been recently constructed in \cite{anev}. Here we give an
alternative construction.
\end{remark}

\subsection{Comparison of the methods over $\mathbb{F}_{2}+u\mathbb{F}_{2}$}

\noindent In this section we compare the methods; four circulant
construction and generalized four circulant construction over $%
\mathbb{F}_{2}+u\mathbb{F}_{2}$ for length $32$. A complete
classification of four circulant codes of length $32$ over
$\mathbb{F}_{2}+u\mathbb{F}_{2}$ is given by Karadeniz et al. in
\cite{karadeniz4}. Four circulant type I codes of length 32 have
binary images corresponding to weight enumerators with $\beta =
0,16,32,48$ and $80$ in $W_{64,2}$. In Table \ref{tab:R1} we provide
generalized four circulant codes. It is observed that the latter
method is more efficient as it produces many more codes of length
$64$ with parameters that could not be obtained from the ordinary
four circulant construction.
\begin{table}[H]
\caption{$[64,32,12]$ codes via Theorem \protect\ref{general} over $%
\mathbb{F}_{2}+u\mathbb{F}_{2}$ ($W_{64,2}$)} \label{tab:R1}
\begin{center}
\scalebox{0.9}{
\begin{tabular}{||c|c|c|c|c||c||}
\hline $\mathcal{F}_{i}$ & $r_{A}$ & $r_{B}$ & $r_{C}$ & $\left\vert
Aut(\mathcal{F}_i)\right\vert $ & $\beta$ \\ \hline\hline ${1}$ &
$(0,u,0,0,1,u,3,0) $ & $(u,u,u,0,0,1,1,3) $ & $(0,u,3,0,0,u,3,0) $ &
$2^{4}$ & $0$\\ \hline $2$ & $(u,u,0,u,1,u,1,u) $ &
$(u,u,u,0,0,1,1,3)$ & $(0,0,3,0,0,0,3,0) $ & $2^{5}$ & $0$\\ \hline
${3}$ & $(u,u,u,u,1,u,3,u) $ & $(u,u,u,0,0,1,1,3) $ &
$(u,u,1,u,u,u,1,u) $ & $2^{6}$ & $0$\\ \hline $4$ &
$(u,0,u,0,0,1,u,3) $ & $(u,u,u,u,0,1,1,1)$ & $(1,u,3,0,3,u,1,0) $ &
$2^{3}$ & $4$\\ \hline ${5}$ & $(u,u,u,u,1,1,3,1) $ &
$(u,u,0,1,0,0,1,3) $ & $(u,u,0,u,u,u,0,u) $ & $2^{4}$ & $4$\\ \hline
$6$ & $(0,u,0,u,1,u,3,u) $ & $(u,u,u,0,0,1,1,3)$ &
$(u,0,3,0,u,0,3,0) $ & $2^{5}$ & $4$\\ \hline ${7}$ &
$(0,u,0,u,1,1,0,1) $ & $(u,u,0,u,1,1,1,3) $ & $(u,0,u,0,u,3,u,3) $ &
$2^{3}$ & $8$\\ \hline $8$ & $(u,0,0,u,1,u,1,0) $ &
$(u,u,u,u,0,1,1,1)$ & $(u,0,1,0,u,0,1,0) $ & $2^{4}$ & $8$\\ \hline
${9}$ & $(u,u,u,u,1,1,u,1) $ & $(u,u,u,u,u,1,0,1) $ &
$(u,0,u,u,1,3,1,1) $ & $2^{5}$ & $8$\\ \hline $10$ &
$(u,u,0,0,1,1,0,3) $ & $(u,u,0,u,1,1,1,3)$ & $(u,u,0,u,u,1,0,1) $ &
$2^{3}$ & $12$\\ \hline ${11}$ & $(u,u,u,u,0,1,0,3) $ &
$(u,u,u,u,0,1,1,1) $ & $(1,u,3,u,1,u,3,u) $ & $2^{4}$ & $12$\\
\hline $12$ & $(0,u,0,u,0,1,u,3) $ & $(u,u,u,0,0,1,1,3)$ &
$(1,0,3,0,1,0,3,0) $ & $2^{5}$ & $12$\\ \hline ${13}$ &
$(u,0,0,0,u,1,u,3) $ & $(u,u,u,0,0,1,1,3) $ & $(1,u,3,u,1,u,3,u) $ &
$2^{6}$ & $12$\\ \hline $14$ & $(u,0,u,0,1,1,u,1) $ &
$(u,u,u,u,u,1,0,1)$ & $(0,0,0,u,1,3,1,1) $ & $2^{4}$ & $16$\\ \hline
${15}$ & $(u,u,u,u,0,1,0,3) $ & $(u,u,u,u,0,1,1,1) $ &
$(3,0,3,0,3,0,3,0) $ & $2^{5}$ & $16$\\ \hline $16$ &
$(u,u,0,u,u,1,u,3) $ & $(u,u,u,0,0,1,1,3)$ & $(3,0,3,0,3,0,3,0) $ &
$2^{6}$ & $16$\\ \hline ${17}$ & $(0,0,u,0,0,1,0,3) $ &
$(u,u,u,0,0,1,1,3) $ & $(3,0,3,0,3,0,3,0) $ & $2^{7}$ & $16$\\
\hline $18$ & $(u,u,u,u,1,1,u,1) $ & $(u,u,0,0,u,1,u,3)$ &
$(u,u,u,0,3,3,3,1) $ & $2^{3}$ & $20$\\ \hline ${19}$ &
$(u,0,0,u,1,u,1,0) $ & $(u,u,u,u,0,1,1,1) $ & $(0,0,1,0,0,0,1,0) $ &
$2^{4}$ & $20$\\ \hline $20$ & $(u,u,u,u,1,1,0,1) $ &
$(u,u,0,u,1,3,3,1)$ & $(u,u,u,u,u,3,0,3) $ & $2^{3}$ & $24$\\ \hline
${21}$ & $(u,0,0,u,0,1,u,1) $ & $(u,u,u,u,0,1,1,1) $ &
$(3,u,3,0,3,u,3,0) $ & $2^{4}$ & $24$\\ \hline $22$ &
$(u,u,0,u,0,1,0,1) $ & $(u,u,u,0,0,u,0,1)$ & $(1,0,1,0,1,3,1,3) $ &
$2^{5}$ & $24$\\ \hline ${23}$ & $(u,0,0,0,u,1,u,3) $ &
$(u,u,u,0,0,1,1,3) $ & $(1,u,3,0,1,u,3,0) $ & $2^{4}$ & $28$\\
\hline $24$ & $(u,u,u,u,u,1,0,3) $ & $(u,u,u,0,0,1,1,3)$ &
$(1,u,1,0,1,u,1,0) $ & $2^{5}$ & $32$\\ \hline ${25}$ &
$(u,u,0,0,1,0,1,u) $ & $(u,u,u,u,0,1,1,1) $ & $(u,0,3,0,u,0,3,0) $ &
$2^{4}$ & $36$\\ \hline $26$ & $(u,u,u,0,1,u,3,0) $ &
$(u,u,u,0,0,1,1,3)$ & $(u,u,3,u,u,u,3,u) $ & $2^{5}$ & $36$\\ \hline
${27}$ & $(0,u,u,0,0,1,0,1) $ & $(u,u,u,0,0,1,1,3) $ &
$(1,0,3,0,1,0,3,0) $ & $2^{4}$ & $44$\\ \hline $28$ &
$(0,u,u,0,0,1,0,1) $ & $(u,u,u,0,0,1,1,3)$ & $(1,u,1,u,1,u,1,u) $ &
$2^{5}$ & $48$\\ \hline ${29}$ & $(u,0,u,0,u,1,0,3) $ &
$(u,u,u,0,0,1,1,3) $ & $(3,u,3,u,3,u,3,u) $ & $2^{7}$ & $80$\\
\hline
\end{tabular}}
\end{center}
\end{table}

\subsection{Computational results over $\mathbb{F}_{4}+u\mathbb{F}_{4}$}
Four circulant codes of length 16 over
$\mathbb{F}_{4}+u\mathbb{F}_{4}$ had been studied in \cite{kaya}.
The binary images of these codes are Type I extremal self-dual
binary codes with weight enumerators
$\beta=0,4,8,12,24,28,32,36,40,48$ and $52$ in $W_{64,2}$. We apply
Corollary \ref{symmetric} and observe that it provides many new
parameters that could not be constructed from the four circulant
construction. The results are tabulated in Table
\ref{tab:tableF4cor1}.
\begin{center}
\begin{table}[H]
\caption{Self-dual codes via Corollary \protect\ref%
{symmetric} over $\mathbb{F}_{4}+u\mathbb{F}_{4}$ of length 16 whose
binary images are self-dual codes of length 64}
\label{tab:tableF4cor1}
\begin{center}
\begin{tabular}{||c|c|c|c|c||c||}
\hline
$\mathcal{E}_{i}$ & $r_{A}$ & $r_{B}$ & $r_{C}$ & $\left\vert Aut(\mathcal{E}%
_{i})\right\vert $ & $\beta $ in $W_{64,2}$ \\ \hline\hline
$\mathcal{E}_{1}$ & $\left( D,F,5,F\right) $ & $\left( E,C,0,1\right) $ & $%
\left( 7,B,4,A\right) $ & $2^{3}$ & $\mathbf{1}$ \\ \hline
$\mathcal{E}_{2}$ & $\left( 5,B,D,B\right) $ & $\left( A,E,B,D\right) $ & $%
\left( 7,F,1,8\right) $ & $2^{3}$ & $\mathbf{5}$ \\ \hline
$\mathcal{E}_{3}$ & $\left( B,9,D,9\right) $ & $\left( 2,E,9,7\right) $ & $%
\left( D,7,1,2\right) $ & $2^{3}$ & $9$ \\ \hline
$\mathcal{E}_{4}$ & $\left( D,F,F,F\right) $ & $\left( E,E,9,8\right) $ & $%
\left( A,6,9,7\right) $ & $2^{3}$ & $\mathbf{13}$ \\ \hline
$\mathcal{E}_{5}$ & $\left( 9,7,7,7\right) $ & $\left( 0,F,C,0\right) $ & $%
\left( 4,1,F,A\right) $ & $2^{3}$ & $\mathbf{17}$ \\ \hline
$\mathcal{E}_{6}$ & $\left( F,9,7,9\right) $ & $\left( 2,4,3,F\right) $ & $%
\left( F,5,B,8\right) $ & $2^{3}$ & $\mathbf{21}$ \\ \hline
$\mathcal{E}_{7}$ & $\left( D,0,F,0\right) $ & $\left( 9,E,C,A\right) $ & $%
\left( D,B,4,2\right) $ & $2^{3}$ & $\mathbf{29}$ \\ \hline
$\mathcal{E}_{8}$ & $\left( 5,8,6,8\right) $ & $\left( F,5,E,8\right) $ & $%
\left( 0,8,9,7\right) $ & $2^{5}$ & $40$ \\ \hline
$\mathcal{E}_{9}$ & $\left( B,4,4,4\right) $ & $\left( 7,E,8,D\right) $ & $%
\left( 0,6,4,2\right) $ & $2^{5}$ & $\mathbf{52}$ \\ \hline
\end{tabular}%
\end{center}
\end{table}
\end{center}

\begin{remark}
The codes with weight enumerators for 1, 5, 13, 17, 21, 29 and 52
were first constructed in \cite{karadeniz64} as $R_{3}$ lifts of the
extended binary Hamming code. These are reconstructed by a circulant
construction in Table \ref{tab:tableF4cor1}.
\end{remark}

\section{New extremal binary self-dual codes of length 68}

The possible weight enumerator of an extremal binary self-dual code
of length 68 (of parameters $\left[ 68,34,12\right]$) is in one of
the following forms by \cite{buyuklieva,harada}:
\begin{eqnarray*}
W_{68,1} &=&1+\left( 442+4\beta \right) y^{12}+\left( 10864-8\beta
\right)
y^{14}+\cdots ,104\leq \beta \leq 1358, \\
W_{68,2} &=&1+\left( 442+4\beta \right) y^{12}+\left( 14960-8\beta
-256\gamma \right) y^{14}+\cdots
\end{eqnarray*}%
where $0\leq \gamma \leq 9$. Recently, Yankov et al. constructed the
first examples of codes with a weight enumerator for $\gamma =7$ in
$W_{68,2}$. Together with these, the existence of codes in
$W_{68,2}$ is known for many values. In order to save space we only
give the lists for $\gamma =5$ and  $\gamma =6$, which are updated
in this work;
\begin{equation*}
\gamma =5\text{ with }\beta \in
\left\{\text{113,116,$...$,181}\right\}
\gamma =6\text{ with }\beta \in \left\{ 2m|m=\text{69, 77, 78, 79, 81, 88}%
\right\}
\end{equation*}
We construct 36 new codes with the rare parameter $\gamma=6$ and 7
codes with $\gamma=5$ in $W_{68,2}$.

We first construct two new codes of length $68$ by applying the
extension method described in Theorem \ref{extension} over
$\F_2+u\F_2$ to $\mathcal{F}_2$ from Table \ref{tab:R1}.
\begin{table}[H]
\caption{New codes of length $68$ as extensions of codes in Table
\protect \ref{tab:R1} by Theorem \ref{extension}}\label{F4UF4C4}
\par
\begin{center}
\scalebox{1}{
   \begin{tabular}{cccccc} \hline
$\mathcal{D}_{68,i}$&  $\mathcal{F}_i$ & $c$  & $X$ & $\gamma$  &
$\beta$  \\ \hline $\mathcal{D}_{68,1}$ &  $2$ &$1$ &
$(31u011u30uu113u3333u11u010301101)$ & $\textbf{5}$ &
$\textbf{101}$\\ \hline $\mathcal{D}_{68,2}$ &  $2$ &$1+u$ & $(1 3 0
0 3 1 u 3 0 0 0 1 3 1 0 1 3 1 3 u 3 1 u u 3 0 1 u 3 1 0 3)$ &
$\textbf{5}$ & $\textbf{105}$\\ \hline
    \end{tabular}}
\end{center}
\end{table}

Now, applying the neighboring construction to the codes obtained in
Table \ref{F4UF4C4}, we get the following new codes of length 68:

\begin{table}[H]
\caption{New codes of length $68$ as neighbors of codes in Table \protect\ref%
{F4UF4C4}}\label{C4new}
\par
\begin{center}
\scalebox{1}{
   \begin{tabular}{ccccc} \hline
$\mathcal{N}_{68,i}$&  $\mathcal{D}_{68,i}$   & $x$ & $\gamma$  &
$\beta$  \\ \hline $\mathcal{N}_{68,1}$ &  $2$ &  $(1 0 1 0 1 1 1 0
0 1 0 1 1 1 0 0 0 1 0 1 0 0 0 0 0 0 1 0 0 0 0 0 0 0)$ & $\textbf{5}$
& $\textbf{109}$\\ \hline $\mathcal{N}_{68,2}$ &  $2$ &  $(0 0 0 0 1
1 0 0 1 1 0 1 1 0 1 0 1 1 1 1 0 1 1 1 0 1 0 0 0 1 1 1 0 0)$ &
$\textbf{5}$ & $\textbf{111}$\\ \hline $\mathcal{N}_{68,3}$ &  $2$ &
$(0 0 0 0 1 1 1 1 1 0 0 1 1 1 1 1 1 1 1 0 1 1 0 1 0 0 0 1 1 0 0 0 0
0)$ & $\textbf{5}$ & $\textbf{112}$\\ \hline $\mathcal{N}_{68,4}$ &
$2$  & $(1 1 0 1 1 1 0 0 0 0 0 0 1 0 0 1 0 1 1 1 0 0 1 0 1 1 0 0 0 1
0 1 0 1)$ & $\textbf{5}$ & $\textbf{114}$\\ \hline
$\mathcal{N}_{68,5}$ &  $1$  & $(1 1 1 0 1 1 0 1 0 0 0 0 0 0 0 1 0 0
1 1 0 0 0 0 0 1 1 1 0 0 1 0 1 0)$ & $\textbf{5}$ & $\textbf{115}$\\
\hline $\mathcal{N}_{68,6}$ &  $2$  & $(1 0 0 1 0 1 0 0 0 1 0 1 0 1
0 1 1 1 0 0 1 0 1 1 1 1 1 0 1 1 1 0 0 0)$ & $\textbf{6}$ &
$\textbf{133}$\\ \hline
    \end{tabular}}
\end{center}
\end{table}

\begin{example}
\label{gamma6}Let $\mathcal{C}_{68}$ be the code obtained by extending $%
\varphi \left( \mathcal{E}_{6}\right) $ over
$\mathbb{F}_{2}+u\mathbb{F}_{2}$ where $X=\left(
3,0,1,1,u,0,u,3,0,1,1,u,1,0,0,u,3,0,0,u,u,u,1,3,3,0,1,3,1,u,0,3\right)$,
then the binary image of $\mathcal{C}_{68}$ is an extremal self-dual $\left[ 68,34,12%
\right]$ code with weight enumerator for $\gamma =6$ and $\beta
=157$ in $W_{68,2}$. As listed above only $6$ codes with $\gamma =6$
were known before. So this is the first example of a self-dual code
with the corresponding weight enumerator.
\end{example}

Without loss of generality, we consider the standard form of the
generator
matrix of $\varphi \left( \mathcal{C}_{68}\right) $. Let $x\in {\mathbb{F}}%
_{2}^{68}-\varphi \left( \mathcal{C}_{68}\right) $ then
$D=\left\langle
\left\langle x\right\rangle ^{\bot }\cap \varphi \left( \mathcal{C}%
_{68}\right) ,x\right\rangle $ is a neighbor of $\varphi \left( \mathcal{C}%
_{68}\right) $. The first $34$ entries of $x$ are set to be $0$, the
rest of the vectors are listed in Table \ref{tab:68}. As neighbors
of $\varphi \left( \mathcal{C}_{68}\right) $\ we obtain $34$ new
codes with weight
enumerators for $\gamma =6$ in $W_{68,2},$ which are listed in Table \ref%
{tab:68}. All the codes have an automorphism group of order 2.

\begin{table}[H]
\caption{New extremal binary self-dual codes of length $68$ with $\protect%
\gamma =6$ as neighbors of $\mathcal{C}_{68}$}
\label{tab:68}%
\scalebox{0.8}{
\begin{tabular}{||c|c||c||c|c||c||}
\hline $\mathcal{C}_{68,i}$ & $X$ & $\beta $ &$\mathcal{C}_{68,i}$ &
$X$ & $\beta $\\ \hline\hline $\mathcal{C}_{68,1}$ & $\left(
1111111100111100001100001000000111\right) $ & 137 &
$\mathcal{C}_{68,2}$ & $\left(
0101001001001111111011100010111011\right) $ & 139 \\ \hline
$\mathcal{C}_{68,3}$ & $\left(
1000001100000110110110000111100010\right) $ & 140 &
$\mathcal{C}_{68,4}$ & $\left(
0010011101110110011001110110110110\right) $ & 141 \\ \hline
$\mathcal{C}_{68,5}$ & $\left(
1111111111000011111101100010011001\right) $ & 142 &
$\mathcal{C}_{68,6}$ & $\left(
1001000001111111111010010000011110\right) $ & 143 \\ \hline
$\mathcal{C}_{68,7}$ & $\left(
1100100010000111001100111111110001\right) $ & 144 &
$\mathcal{C}_{68,8}$ & $\left(
0000110001110110011011011010000110\right) $ & 145 \\ \hline
$\mathcal{C}_{68,9}$ & $\left(
1000010100001010110101110111110101\right) $ & 146 &
$\mathcal{C}_{68,10}$ & $\left(
1100110100000010010000110010011110\right) $ & 147 \\ \hline
$\mathcal{C}_{68,11}$ & $\left(
1110101000011110100101111111101011\right) $ & 148 &
$\mathcal{C}_{68,12}$ & $\left(
0110011001001101000111010101011000\right) $ & 149 \\ \hline
$\mathcal{C}_{68,13}$ & $\left(
1111111100101101000000001011111000\right) $ & 150 &
$\mathcal{C}_{68,14}$ & $\left(
0000100001100010111010011111111000\right) $ & 151 \\ \hline
$\mathcal{C}_{68,15}$ & $\left(
1110000010100000001110110110000101\right) $ & 152 &
$\mathcal{C}_{68,16}$ & $\left(
1010100100110011111101001101001001\right) $ & 153 \\ \hline
$\mathcal{C}_{68,17}$ & $\left(
1111010010000100100000101000011101\right) $ & 155 &
$\mathcal{C}_{68,18}$ & $\left(
1000001011110111100101100000001000\right) $ & 159 \\ \hline
$\mathcal{C}_{68,19}$ & $\left(
0001010001010101010010010001100010\right) $ & 160 &
$\mathcal{C}_{68,20}$ & $\left(
1100000100011110101111110001010101\right) $ & 161 \\ \hline
$\mathcal{C}_{68,21}$ & $\left(
0101110011110010110000111111010011\right) $ & 163 &
$\mathcal{C}_{68,22}$ & $\left(
1000000111111000000010111100010001\right) $ & 164 \\ \hline
$\mathcal{C}_{68,23}$ & $\left(
0100000001010000001001110110010110\right) $ & 165 &
$\mathcal{C}_{68,24}$ & $\left(
0111001010010100000010010010101000\right) $ & 166 \\ \hline
$\mathcal{C}_{68,25}$ & $\left(
1111010011000111000101101001011100\right) $ & 167 &
$\mathcal{C}_{68,26}$ & $\left(
0010110010110100000010001111000000\right) $ & 168 \\ \hline
$\mathcal{C}_{68,27}$ & $\left(
0000010011010110001010010000101001\right) $ & 169 &
$\mathcal{C}_{68,28}$ & $\left(
1110101000110000011111010101010101\right) $ & 170 \\ \hline
$\mathcal{C}_{68,29}$ & $\left(
1110100001100111100100000010010010\right) $ & 171 &
$\mathcal{C}_{68,30}$ & $\left(
1000001101101110010001101010111101\right) $ & 172 \\ \hline
$\mathcal{C}_{68,31}$ & $\left(
1100100001110011101001010001100000\right) $ & 173 &
$\mathcal{C}_{68,32}$ & $\left(
0100011001000011000100010100101100\right) $ & 174 \\ \hline
$\mathcal{C}_{68,33}$ & $\left(
0110000001110110000111101000101011\right) $ & 177 &
$\mathcal{C}_{68,34}$& $\left( 1 0 1 1 1 1 1 0 0 0 1 0 0 0 0 0 0 0 1
0 1 1 0 1 0 0 0 0 1 0 1 0 1 0\right) $ & 184\\ \hline
\end{tabular}}
\end{table}

\section{Conclusion}

In this paper, we generalize the well known four circulant
construction for constructing self-dual codes. We compare both
methods to highlight the significance of the generalized
construction. Additionally, we construct many new codes of length
$68$. For codes of length $68$, we constructed the following codes
with new weight enumerators in $W_{68,2}$:

\begin{equation*}
\begin{split}
\gamma =5,& \quad \beta =\{101,105,109,111,112,114,115\}. \\
\gamma =6,& \quad \beta
=\{133,137,139,140,141,142,143,144,145,146,147,148,149,150,151,152, \\
& \qquad \quad
153,155,157,159,160,161,163,164,165,166,167,168,169,170,171,
\\
& \qquad \quad 172,173,174,177,184\}.
\end{split}%
\end{equation*}

The binary generator matrices of the new codes we have constructed
are available online at \cite{web}.

The results we have obtained have demonstrated the effectiveness of
the new construction and the difference from the ordinary
four-circulant construction. A possible direction for future
research could be applying these constructions for different rings
and lengths.

\end{document}